\documentclass[leqno]{article}

\usepackage[frenchb,english]{babel}
\usepackage[utf8]{inputenc}
\usepackage{amsmath}
\usepackage{amssymb}
\usepackage{amsfonts}
\usepackage{enumerate}
\usepackage{vmargin}
\usepackage[all]{xy}
\usepackage{mathrsfs}
\usepackage{mathtools}
\usepackage{lmodern}
\usepackage{slashed}
\usepackage[colorlinks=true,linkcolor=blue,pagebackref=true]{hyperref}%
\setmarginsrb{3cm}{3cm}{3.5cm}{3cm}{0cm}{0cm}{1.5cm}{3cm}
\usepackage{comment}

\usepackage{accents}

\newlength{\dhatheight}

\newcommand{\A}{\mathrm{A}}

\newcommand{\B}{\mathrm{B}} 

\newcommand{\I}{\mathrm{I}} 

\let\L\relax 
\newcommand{\L}{\mathrm{L}}

\newcommand{\M}{\mathrm{M}}

\let\P\relax 
\newcommand{\P}{\mathrm{P}} 

\let\cal\relax
\newcommand{\cal}{\mathcal}
\newcommand{\Z}{\mathrm{Z}}

\newcommand{\VN}{\mathrm{VN}}
\newcommand{\la}{\langle}
\newcommand{\ra}{\rangle}

\newcommand{\conv}{\mathrm{conv}}
\newcommand{\Proj}{\mathrm{Proj}}

\renewcommand{\leq}{\ensuremath{\leqslant}}
\renewcommand{\geq}{\ensuremath{\geqslant}}
\newcommand{\qed}{\hfill \vrule height6pt  width6pt depth0pt}

\newcommand{\norm}[1]{\left\Vert#1\right\Vert}

\newcommand{\co}{\colon}

\newcommand{\ovl}{\overline}

\newcommand{\ov}{\overset}

\newcommand{\Face}{\mathrm{Face}}

\newcommand{\rank}{\mathrm{rank}}

\renewcommand{\d}{\mathop{}\mathopen{}\mathrm{d}} 

\DeclareMathOperator{\Span}{span} 
\DeclareMathOperator{\Irr}{Irr} 

\newtheorem{thm}{Theorem}[section]

\newtheorem{prop}[thm]{Proposition}

\newtheorem{cor}[thm]{Corollary}
\newtheorem{lemma}[thm]{Lemma}

\newtheorem{remark}[thm]{Remark}
\newtheorem{example}[thm]{Example}

\newenvironment{proof}[1][]{\noindent {\it Proof #1} : }{\hbox{~}\qed
\smallskip
}

\usepackage{tocloft}
\numberwithin{equation}{section}
\usepackage[nottoc,notlot,notlof]{tocbibind}

\let\OLDthebibliography\thebibliography
\renewcommand\thebibliography[1]{
  \OLDthebibliography{#1}
  \setlength{\parskip}{0pt}
  \setlength{\itemsep}{0pt plus 0.3ex}
}

\newcommand\reallywidehat[1]{\arraycolsep=0pt\relax%
\begin{array}{c}
\stretchto{
  \scaleto{
    \scalerel*[\widthof{\ensuremath{#1}}]{\kern-.5pt\bigwedge\kern-.5pt}
    {\rule[-\textheight/2]{1ex}{\textheight}} 
  }{\textheight} %
}{0.5ex}\\           
#1\\                 
\rule{-1ex}{0ex}
\end{array}
}

\begin{document}
\selectlanguage{english}
\title{\bfseries{On the convex structure of the space of quantum channels which act as Fourier multipliers}}
\date{}
\author{\bfseries{C\'edric Arhancet and Lei Li}}
\maketitle


\begin{abstract}
If $G$ is a compact group, continuous normalized positive definite functions are in one-to-one correspondence with unital quantum channels acting as Fourier multipliers on the group von Neumann algebra $\mathrm{VN}(G)$. We study the geometry of the convex set $\mathrm{P}_1(G)$ of normalized positive definite functions, equipped with the topology induced by the norm topology of the Fourier algebra $\mathrm{A}(G)$, and its relation with the structure of $\mathrm{VN}(G)$. We show that the von Neumann algebras of two compact groups $G$ and $H$ are $*$-isomorphic if and only if the convex sets $\mathrm{P}_1(G)$ and $\mathrm{P}_1(H)$ are affinely homeomorphic. We also describe the group of affine homeomorphisms of $\mathrm{P}_1(G)$ in terms of Jordan $*$-automorphisms of $\mathrm{VN}(G)$.
\end{abstract}


\makeatletter
 \renewcommand{\@makefntext}[1]{#1}
 \makeatother

{
  \hypersetup{linkcolor=blue}
 \tableofcontents
}


\section{Introduction}
\label{sec:Introduction}

Quantum channels are the basic morphisms of quantum information theory. In the Heisenberg picture, a quantum channel is a weak* continuous unital completely positive map between von Neumann algebras. In this paper, we focus on a natural subclass of quantum channels arising from harmonic analysis, namely those channels which act as Fourier multipliers on group von Neumann algebras.

If $G$ is a locally compact group equipped with a left Haar measure, recall that the group von Neumann algebra $\VN(G)$ is the von Neumann algebra generated by the range $\lambda(G)$ of the left regular representation $\lambda$ of $G$ on the complex Hilbert space $\L^2(G)$. Note that for any $s \in G$ the left translation operator $\lambda_s \co \L^2(G) \to \L^2(G)$ is defined by $(\lambda_s \xi)(t) \ov{\mathrm{def}}{=} \xi(s^{-1}t)$, where $\xi \in \L^2(G)$ and $s,t \in G$. Observe that the subspace $\Span \{\lambda_s : s \in G\}$ is weak* dense in the von Neumann algebra $\VN(G)$. Moreover, if the group $G$ is abelian, then the von Neumann algebra $\VN(G)$ is $*$-isomorphic to the algebra $\L^\infty(\hat{G})$ of essentially bounded functions on the Pontryagin dual $\hat{G}$ of $G$. As basic models of quantum groups, these algebras play a fundamental role in operator algebras. We refer to \cite{BrO08} and \cite{KaL18} for more information.

Given a continuous positive definite function $\varphi \co G \to \mathbb{C}$, one can define, according to \cite[Proposition 5.4.9, p.~184]{KaL18}, a weak* continuous completely positive map  $M_\varphi \co \VN(G) \to \VN(G)$ by
\[
M_\varphi(\lambda_s)
\ov{\mathrm{def}}{=} \varphi(s) \lambda_s,
\quad s \in G.
\]
This map is unital precisely when $\varphi(e)=1$. Such unital maps are called Fourier multiplier quantum channels associated with $G$. They constitute a fundamental class of quantum channels arising from noncommutative harmonic analysis. In \cite{Arh25}, Fourier multiplier channels were investigated, with a particular emphasis on finite groups, and exhaustive descriptions are obtained for concrete examples, such as the quaternion group $\mathbb{Q}_8$ and the dihedral group $\mathbb{D}_4$. As a byproduct, it is shown, remarkably, that each quantum channel acting on the matrix algebra $\M_2$ is unitarily equivalent to the compression of a Fourier multiplier on one of these groups.

Guided by the motivation to understand the differences between the structures of the spaces of Fourier multipliers on groups such as $\mathbb{Q}_8$ and $\mathbb{D}_4$, the aim of the present paper is to study the convex geometry of the set of Fourier multiplier channels through the convex geometry of positive definite functions when $G$ is compact. Let $\P_1(G)$ be the set of continuous positive definite functions $\varphi$ on $G$ such that $\varphi(e)=1$. The bijective correspondence $\varphi \mapsto M_\varphi$ is affine, and it turns the study of convex decompositions of Fourier multiplier channels into the study of convex decompositions inside $\P_1(G)$. 

Our first observation is that the relevant convex set is more naturally described in terms of normal states of $\VN(G)$. Recall that Eymard's identification yields an isometric isomorphism $\A(G) \cong \VN(G)_*$ between the Fourier algebra $\A(G)$ of the locally compact group $G$ and  the predual $\VN(G)_*$ of the von Neumann algebra $\VN(G)$. We show that $\A(G) \cap \P_1(G)$ can be identified affinely and homeomorphically with the normal state space of $\VN(G)$. It follows that a $*$-isomorphism $\VN(G) \cong \VN(H)$ induces an affine homeomorphism between $\A(G) \cap \P_1(G)$ and $\A(H) \cap \P_1(H)$.

We then concentrate on compact groups. In this case, we have $\P_1(G)=\A(G)\cap \P_1(G)$. We equip $\P_1(G)$ with the induced topology coming from the norm topology of $\A(G)$, which identifies $\P_1(G)$ with the normal state space of $\VN(G)$. This allows us to use the general theory of facial structures of state spaces of operator algebras, in particular the correspondence between central projections and split faces developed in \cite{AlS76} and \cite{AlS01}. Our main result states that for compact groups $G$ and $H$, the von Neumann algebras $\VN(G)$ and $\VN(H)$ are $*$-isomorphic if and only if the convex sets $\P_1(G)$ and $\P_1(H)$ are affinely homeomorphic.

The present approach is related in spirit to the work \cite{LWW22}, where norm-one positive definite functions are used to extract structural information about finite groups. Our results provide both a generalization and a converse characterization, completing the result \cite[Theorem 4.6 p.~99]{LWW22}. Finally, in the last section, we investigate the group of affine homeomorphisms of the set $\P_1(G)$ of normalized positive definite functions in the case where $G$ is compact. We give a very concrete description, via normal Jordan $*$-automorphisms of $\VN(G)$.

\paragraph{Structure of the paper.}
In Section \ref{Sec-preliminaries}, we recall basic notions from convex geometry and the correspondence between projections in a von Neumann algebra and faces of its normal state space. In Section \ref{sec-type-I}, we identify the set $\A(G)\cap \P_1(G)$ with the normal state space of $\VN(G)$ and we explain how $*$-isomorphisms of group von Neumann algebras induce affine homeomorphisms at the level of positive definite functions. In Section \ref{sec-compact}, we specialize to compact groups and show how the convex structure of $\P_1(G)$ determines the group von Neumann algebra $\VN(G)$. Finally, in Section \ref{sec-affine-homeo}, we describe the group of affine homeomorphisms of the set $\P_1(G)$ of normalized positive definite functions in the case where the group $G$ is compact.

\section{Preliminaries}
\label{Sec-preliminaries}
\paragraph{Geometry}

Let $C$ be a convex subset of a real vector space. A point $x \in C$ is said to be extreme if it cannot be written in a nontrivial way as a convex combination
of points of $C$, i.e., if for any $t \in (0,1)$ and $y, z \in C$ the equality
\begin{equation}
\label{def-extreme}
x = t y+(1-t)z \quad \text{implies} \quad x=y=z.
\end{equation}
Following \cite[p.~1]{AlS01}, we say that a convex subset $F$ of a convex set $C$ is a face of $C$ if the following implication holds for any $x,y \in C$ and any $t \in (0,1)$
\begin{equation}
\label{def-face}
t x+(1-t)y \in F \quad \text{implies} \quad x,y \in F.
\end{equation}

Recall that two convex subsets $F$ and $G$ of a real vector space $X$ are said to be affinely independent \cite[p.~6]{AlS01} if every point $z$ in their convex hull $\conv (F \cup G)$ can be uniquely expressed as a convex combination
$$
z 
= t x + (1-t)y,
$$
where $0 \leq t \leq 1$, $x \in F$ and $y \in G$. Here ``uniqueness'' means uniqueness up to the obvious indeterminacy of $x$ when $t = 0$ and of $y$ when $t= 1$.

Following \cite[Definition 1.4 p.~6]{AlS01}, we say that a convex set $C$ is the free convex sum of two convex subsets $F$ and $G$ if $C=\conv(F \cup G)$ and if $F$ and $G$ are affinely independent. In this case, we write 
$$
C
=F \oplus_c G.
$$
It is known that if $C = F \oplus_c G$ then the two sets $F$ and $G$ must be faces of $C$. We say that a face $F$ of $C$ is a split face if there exists another face $G$ such that $C = F \oplus_c G$. In this case, $G$ is unique. We call it the complementary split face of $F$.

\paragraph{Affine maps} Let $C$ and $D$ be convex subsets of real vector spaces. Following \cite[p.~2]{AlS01}, a map $T \co C \to D$ is said to be affine if it preserves convex combinations, i.e., 
$$
T(tx + (1 - t)y) 
= t T(x) + (1 - t)T(y), \quad x, y \in C,t \in [0,1].
$$ 
If $C$ and $D$ are convex subsets of real normed spaces, an affine homeomorphism of $C$ onto $D$ is a homeomorphism of $C$ onto $D$ which is an affine map. We will use the following elementary lemma.

\begin{lemma}
\label{lem-affine-homeomorphism-preserves-face-chains}
Let $C$ and $D$ be closed convex subsets of real normed spaces and let $T \co C \to D$ be an affine homeomorphism. Then the map $F \mapsto T(F)$ is an order isomorphism between the families of norm-closed faces of $C$ and $D$. In particular, for every norm-closed face $F$ of $C$, the map $T$ induces a bijection between the strictly increasing chains of nonempty norm-closed faces contained in $F$ and those contained in $T(F)$, preserving their cardinalities and their maximality. Moreover, $T$ maps split faces onto split faces. More precisely, if $C=F \oplus_c G$ then $D=T(F) \oplus_c T(G)$, so that $T(G)$ is the complementary split face of $T(F)$.
\end{lemma}

\begin{proof}
Let $F$ be a face of the convex set $C$. Suppose that $y_1,y_2 \in D$ and $t \in (0,1)$ satisfy
\[
ty_1+(1-t)y_2 \in T(F).
\]
Since every bijective affine map between convex sets has an affine inverse, the map $T^{-1} \co D \to C$ is affine. We infer that
\[
tT^{-1}(y_1)+(1-t)T^{-1}(y_2)
=
T^{-1}\bigl(ty_1+(1-t)y_2\bigr)
\in F.
\]
Since $F$ is a face of $C$, it follows that $T^{-1}(y_1) \in F$ and $T^{-1}(y_2) \in F$. Hence $y_1$ and $y_2$ belong to $T(F)$. Thus $T(F)$ is a face of $D$. 

Moreover, $T(F)$ is norm-closed since $T$ is a homeomorphism. Applying the same argument to $T^{-1}$ shows that $F\mapsto T(F)$ is a bijection between the norm-closed faces of
$C$ and those of $D$.

Since $T$ is injective, for any subsets $F_1,F_2$ of the convex set $C$, we have
\[
F_1\subsetneq F_2
\quad\Longleftrightarrow\quad
T(F_1)\subsetneq T(F_2).
\]
Therefore $T$ and $T^{-1}$ establish mutually inverse correspondences between strictly increasing chains of nonempty norm-closed faces, preserving their cardinalities. They also preserve maximality of such chains.

Let $F$ be a split face of $C$ and let $G$ be its complementary split face. Thus $
C = F \oplus_c G$. Since $T$ is affine and surjective, we have
\[
D
=T(C)
=T\bigl(\conv(F \cup G)\bigr)
=\conv\bigl(T(F) \cup T(G)\bigr).
\]
We show that $T(F)$ and $T(G)$ are affinely independent. Suppose that
\[
t y_1+(1-t)y_2
=
s z_1+(1-s)z_2,
\]
where $t,s \in [0,1]$, with $y_1,z_1 \in T(F)$ and $y_2,z_2 \in T(G)$. Since $T^{-1}$ is affine, applying it to the preceding equality gives
\[
tT^{-1}(y_1)
+(1-t)T^{-1}(y_2)
=
sT^{-1}(z_1)
+(1-s)T^{-1}(z_2).
\]
Now, we have $T^{-1}(y_1),T^{-1}(z_1) \in F$ and $T^{-1}(y_2), T^{-1}(z_2) \in G$. Since $F$ and $G$ are affinely independent, we obtain $
t=s$, $
T^{-1}(y_1)
=T^{-1}(z_1)$ and $T^{-1}(y_2)=T^{-1}(z_2)$. The injectivity of $T$ implies that $y_1=z_1$ and $y_2=z_2$. Consequently, we have $
D
=T(F)
\oplus_c T(G)$. So $T(F)$ is a split face. Applying the same argument to $T^{-1}$ shows that $T$ sends split faces onto split faces.
\end{proof}

\paragraph{States of operator algebras} 
We describe the correspondence between central projections and their split faces in the following result, which is based on \cite[Theorem 3.35 p.~143]{AlS01} and \cite[Proposition 3.40 p.~148]{AlS01}. See also \cite[Corollary 5.34 p.~162]{AlS03} and \cite[Theorem 11.5 p.~97]{AlS76}. Recall that the normal state space $\cal{K}$ of a von Neumann algebra $\cal{M}$ is the set of normal states on $\cal{M}$.

\begin{prop}
\label{prop-Alfsen}
Let $\cal{M}$ be a von Neumann algebra with normal state space $\cal{K}$. We denote by $\cal{F}$ the set of all norm-closed faces of $\cal{K}$ and by $\Proj(\cal{M})$ the set of all orthogonal projections in $\cal{M}$, 
each equipped with its natural ordering. Then there exists a bijection $\Phi \co \Proj(\cal{M}) \to \cal{F}$, $p \mapsto \Face(p)$, where 
\begin{equation}
\label{def-Face-p}
\Face(p)
\ov{\mathrm{def}}{=} \{\omega \in \cal{K} : \la p,\omega \ra_{\cal{M},\cal{M}_*} = 1\}.
\end{equation}
Moreover, this correspondence preserves the order, i.e., for any $p,q \in \Proj(\cal{M})$
\begin{equation}
\label{order}
p \leq q 
\Longleftrightarrow \Face(p) \subset \Face(q),
\end{equation}
and 
\begin{equation}
\label{disjoint}
\Face(p \wedge q)
=\Face(p) \cap \Face(q).
\end{equation}
Furthermore, $p$ is a central projection in $\cal{M}$ if and only if $\Face(p)$ is a split face of $\cal{K}$. In this case, the complementary split face of $\Face(p)$ is $\Face(1-p)$.
\end{prop}
 
In this context, we say that $\Face(p)$ is the face supported by the projection $p$. Note that $\Face(0)=\emptyset$ and $\Face(1)=\cal{K}$. 

\paragraph{Projections} 
First, we will use the following elementary result \cite[Proposition 2.5.4 p.~55]{Sun97}.

\begin{prop}
\label{prop-sum-projections}
Suppose that $(H_i)_{i \in I}$ is a family of closed subspaces of a Hilbert space $H$, which are pairwise orthogonal, with associated orthogonal projections $(p_i)_{i \in I}$. Then the family $(p_i)_{i \in I}$ is unconditionally summable, with respect to the strong operator topology. Moreover, the sum $\sum_{i \in I} p_i$ is the orthogonal projection on the closed subspace spanned by $\bigcup_{i \in I} H_i$.
\end{prop}

Consider an orthogonal projection $p$ in a von Neumann algebra $\cal{M}$. Following 
\cite[Definition 3.28 p.~89]{AlS03}, we say that $p$ is minimal if $p \not=0$ and if every orthogonal projection $q$ in $\cal{M}$ such that $q \leq p$ is equal to 0 or to $p$. Equivalently, according to \cite[Proposition 6.4.3 p.~420]{KaR97b} $p$ is minimal if and only if we have the equality $p\cal{M}p = \mathbb{C}p$. By \cite[p.~420]{KaR97b}, if $p$ belongs to the center $\Z(\cal{M})$ and is minimal in $\Z(\cal{M})$ then the reduced von Neumann algebra $\cal{M}p$ is a factor. 

\begin{example} \normalfont
\label{example-minimal}
Consider the von Neumann algebra $\B(H)$ of all bounded operators on a complex Hilbert space $H$. By \cite[Example 5.1.7 p.~309]{KaR97a}, the minimal projections are the orthogonal projections on one-dimensional subspaces of $H$.
\end{example}

\begin{example} \normalfont
\label{example-minimal-2}
If $I$ is a set, consider the abelian von Neumann algebra $\ell^\infty_I$ acting on the Hilbert space $\ell^2_I$. By \cite[Lemma 8.6.8 p.~557]{KaR97b} and its proof, the minimal projections are the elements $e_i$, where $i \in I$.
\end{example}

By \cite[Corollary 6.5.3 p.~424]{KaR97b}, 
 a factor is of type I if and only if it contains a minimal projection. Moreover, the same result states that a factor is of type $\I_n$ if and only if $1$ is the sum of $n$ minimal projections.


A von Neumann algebra $\cal{M}$ is said to be atomic if for every non-zero projection $p$ there exists a minimal projection $q$ such that $q \leq p$. The following characterization is \cite[Lemma 3.9 p.~93]{Hel09} (see also \cite[p.~354]{Bla06} and \cite[Exercise 6.9.37 p.~450]{KaR97b}).

\begin{prop}
\label{prop-carac-atomic}
Let $\cal{M}$ be a von Neumann algebra with center $\Z(\cal{M})$. The following conditions are equivalent.

\begin{enumerate}
	\item $\cal{M}$ is atomic.
	
	\item There exists an orthogonal family $(p_i)_{i \in I}$ of minimal projections of $\cal{M}$ such that $1=\sum_{i \in I} p_i$ in the strong operator topology.
	
	\item $\cal{M}$ is of type $\I$ and there exists an orthogonal family $(p_i)_{i \in I}$ of minimal projections of $\Z(\cal{M})$ such that $1=\sum_{i \in I} p_i$ in the strong operator topology.
	
	\item $\cal{M}$ is a direct sum of type $\I$ factors.
\end{enumerate}
\end{prop}



\paragraph{Harmonic analysis}
Let $G$ be a locally compact group. Recall \cite[p.~286]{Dix77} \cite[Definition 1.4.15 p.~22]{KaL18} that a function $f \co G \to \mathbb{C}$ is said to be positive definite if 
\begin{equation}
\label{def-def-pos}
\sum_{j,k=1}^{n}  \ovl{\xi_k} \xi_j f(s_k^{-1}s_j) \geq 0
\end{equation}
for any $\xi_1,\ldots,\xi_n \in \mathbb{C}$ and any $s_1,\ldots,s_n \in G$. Let $\P_1(G)$ be the set of continuous positive definite functions $\varphi$ on $G$ such that $\varphi(e) = 1$. The sum of two positive definite continuous functions is positive definite. If $\varphi$ is a positive definite function and if $\lambda \geq 0$, then $\lambda\varphi$ is also a positive definite function. In particular, the set $\P_1(G)$ is convex.


\section{A description of the normal state space of a group von Neumann algebra}
\label{sec-type-I}

Let $G$ be a locally compact group equipped with a left Haar measure. According to \cite[Theorem 2.4.3 p.~59]{KaL18} and \cite[Lemma 2.3.6 p.~52]{KaL18}, the Fourier algebra $\A(G)$ can be seen as the set of all functions $\ovl{f}*\check{g}$, where $f,g \in \L^2(G)$. Here $\check{g}(s) \ov{\mathrm{def}}{=} g(s^{-1})$ for any $s \in G$. For any $s \in G$, note that
\begin{equation}
\label{convol-utile}
(\ovl{f}*\check{g})(s)
=\int_G \ovl{f(t)}\check{g}(t^{-1}s) \d t
=\int_G \ovl{f(t)} g(s^{-1}t) \d t
=\la f,\lambda_s g \ra_{\L^2(G)},
\end{equation}
where $\la f,g \ra_{\L^2(G)} \ov{\mathrm{def}}{=} \int_G \ovl{f(t)} g(t) \d t$. In other words, the Fourier algebra $\A(G)$ is the set of matrix coefficients of the left regular representation $\lambda \co G \to \B(\L^2(G))$ of the locally compact group $G$.

We denote by $\VN(G)_*$ the predual of the von Neumann algebra $\VN(G)$, namely the Banach space of all normal linear functionals on $\VN(G)$. The positive cone of the predual $\VN(G)_*$ will be denoted by $\VN(G)_*^+$. It consists precisely of the positive normal linear functionals on the von Neumann algebra $\VN(G)$. Recall that Eymard's identification \cite[Th\'eor\`eme 3.10 p.~210]{Eym64} (see also \cite[Theorem 2.3.9 p.~54]{KaL18}) gives an isometric isomorphism $\Theta_G \co \A(G) \to \VN(G)_*$ of Banach spaces, characterized by the duality bracket
\begin{equation}
\label{duality-Eymard}
\big\la T,\ovl{f}*\check{g} \big\ra_{\VN(G),\mathrm{A}(G)}
=\la f,T(g) \ra_{\L^2(G)}, \quad T \in \VN(G), f,g \in \L^2(G).
\end{equation}
Consider some function $\varphi \in \A(G)$. In particular, writing $\varphi = \ovl{f}*\check{g}$ for some functions $f,g \in \L^2(G)$, we see that for any $s \in G$
\begin{align*}
\MoveEqLeft
\la \lambda_s,\varphi \ra_{\VN(G),\mathrm{A}(G)} 
= \la \lambda_{s}, \ovl{f}*\check{g} \ra_{\VN(G),\A(G)}        
\ov{\eqref{duality-Eymard}}{=} \la  f,\lambda_{s}g \ra_{\L^2(G)} 
\ov{\eqref{convol-utile}}{=}(\ovl{f}*\check{g})(s)
=\varphi(s).
\end{align*}
Consequently, for every function $\varphi \in \A(G)$, there exists a unique normal linear functional $\omega \in \VN(G)_*$ such that
\begin{equation}
\label{Eymard-predual}
\varphi(s)
=\omega(\lambda_s), \quad s \in G.
\end{equation}
This means that $\Theta_G(\varphi)=\omega$.

\begin{prop}
\label{prop-first}
Let $G$ be a locally compact group. The isomorphism $\Theta_G \co \A(G) \to \VN(G)_*$ restricts to an affine homeomorphism
\[
\Theta_G \co \A(G)\cap \P_1(G) \to \{\omega \in \VN(G)_*^+ : \omega(1)=1\},
\]
which allows us to identify the subset $\A(G)\cap \P_1(G)$ with the normal state space $\{\omega \in \VN(G)_*^+ : \omega(1)=1\}$ of the von Neumann algebra $\VN(G)$.
\end{prop}

\begin{proof}
Let $\omega \in \VN(G)_*^+$ with $\omega(1)=1$ and define the function $\varphi_\omega \co G \to \mathbb{C}$ by $\varphi_\omega(s) \ov{\mathrm{def}}{=} \omega(\lambda_s)$, where $s \in G$. For any finitely supported function $\xi \co G \to \mathbb{C}$, we have
\begin{align*}
\MoveEqLeft
\sum_{s,t \in G} \varphi_\omega(t^{-1}s)\ovl{\xi_t}\xi_s
=\sum_{s,t \in G} \omega(\lambda_{t^{-1}s})\ovl{\xi_t}\xi_s 
=\sum_{s,t \in G} \omega(\lambda_{t}^*\lambda_s)\ovl{\xi_t} \xi_s\\
&=\omega\Big[\Big(\sum_{t \in G} \ovl{\xi_t}\lambda_t^*\Big)\Big(\sum_{s \in G} \xi_s \lambda_s\Big)\Big]
=\omega\left(\left(\sum_{t\in G}\xi_t\lambda_t\right)^*\left(\sum_{s \in G} \xi_s\lambda_s\right)\right)
\geq 0.
\end{align*}
since the functional $\omega$ is positive. Hence $\varphi_\omega$ is positive definite. Moreover, we have $\varphi_\omega(e)=\omega(1)=1$. Since $\omega$ is normal, we have $\varphi_\omega \in \A(G)$ and $\Theta_G(\varphi_\omega)=\omega$. Thus $\varphi_\omega \in \A(G)\cap \P_1(G)$.


Conversely, let $\varphi \in \A(G)\cap \P_1(G)$. Consider the linear form $\omega \ov{\mathrm{def}}{=} \Theta_G(\varphi)$ in the predual $\VN(G)_*$. For any finite sum $a=\sum_s a_s \lambda_s$, we compute
\begin{align*}
\MoveEqLeft
\omega(a^*a)
=\omega\Big(\big(\sum_s a_s \lambda_s\big)^*\big(\sum_t a_t \lambda_t\big)\Big) 
=\sum_{s,t \in G} \overline{a_s} a_t \omega(\lambda_{s^{-1}t}) 
\ov{\eqref{Eymard-predual}}{=} \sum_{s,t \in G} \overline{a_s} a_t \varphi(s^{-1}t)
\ov{\eqref{def-def-pos}}{\geq} 0.
\end{align*}
Consequently, $\omega$ is positive on the $*$-algebra $\mathcal{P}_G$ generated by $(\lambda_s)_{s \in G}$, which is weak* dense in $\VN(G)$. If $x$ is a positive element of $\VN(G)$, by Kaplansky Density Theorem \cite[Corollary 5.3.6 p.~329]{KaR97a}, there exists a bounded net $(a_i)$ of positive elements of $\mathcal{P}_G$ such that $a_i \to x^{\frac{1}{2}}$ in the strong operator topology. Consequently, by \cite[Proposition 1.2.1 p.~9]{Li92}, we have $a_i^2 \to x$ in the strong operator topology, since $(a_i)$ is bounded. Since the net $(a_i^2)$ is bounded, we have $a_i^2 \to x$ in the weak* topology, according to \cite[pp.~68-69]{Tak02}. Hence $\omega(x)=\lim_i \omega(a_i^2) \geq 0$.  We conclude that $\omega$ is a positive normal functional on $\VN(G)$. Furthermore, observe that $\omega(1)=\varphi(e)=1$. 
 
Since the map $\Theta_G$ is an isometry, the restriction of $\Theta_G\co \A(G) \to \VN(G)_*$ is an affine homeomorphism for the norm topology of $\A(G)$ and the norm topology of $\VN(G)_*$.
\end{proof}

\begin{prop}
\label{prop-useful}
Let $G$ and $H$ be locally compact groups. Suppose that there exists a $*$-isomorphism $\VN(G) \cong \VN(H)$. Then the convex sets $\A(G)\cap \P_1(G)$ and $\A(H) \cap \P_1(H)$ are affinely homeomorphic, where we use the topologies induced by the norm topologies of $\A(G)$ and $\A(H)$.
\end{prop}

\begin{proof}
Let $\theta \co \VN(G) \to \VN(H)$ be a $*$-isomorphism. Recall that by \cite[Proposition 1.12.3 p.~61]{Li92}, the map $\theta$ is a weak* homeomorphism. Now, consider the preadjoint $\theta_* \co \VN(H)_* \to \VN(G)_*$ defined by
\begin{equation}
\label{predual-map}
\theta_*(\omega)
=\omega \circ \theta, \quad \omega \in \VN(H)_*.
\end{equation}
It is linear, isometric, and maps $\VN(H)_*^+$ onto $\VN(G)_*^+$. Moreover, we have the equivalence
\[
\theta_*(\omega)(1)=1
\ov{\eqref{predual-map}}{\iff} \omega \circ \theta(1)=1
\iff \omega(1)
=1.
\]
Hence $\theta_*$ maps the normal state space of $\VN(H)$ onto the normal state space of $\VN(G)$, affinely and homeomorphically. Therefore, the composition
\[
T
\ov{\mathrm{def}}{=}
\Theta_G^{-1} \circ \theta_* \circ \Theta_H
\co
\A(H) \cap \P_1(H) \to \A(G) \cap \P_1(G)
\]
is an affine homeomorphism. 
\end{proof}

Recall that by definition \cite[Definition 2.1.5 p.~40]{KaL18}, the Fourier-Stieltjes algebra $\B(G)$ of a locally compact group $G$ consists of all finite linear combinations of continuous positive definite functions. It is known that $\A(G)$ is an ideal of $\B(G)$, see \cite[Proposition 2.3.3 p.~51]{KaL18}. According to \cite[p.~209]{Eym64}, if $G$ is compact then we have $\B(G)=\A(G)$. In this particular case, we have 
\begin{equation}
\label{compact-case}
\A(G) \cap \P_1(G)
=\P_1(G).
\end{equation}
From Proposition \ref{prop-useful}, we deduce the following result.

\begin{prop}
\label{prop:compact-converse}
Let $G$ and $H$ be compact groups. Suppose that there exists a $*$-isomorphism $\VN(G) \cong \VN(H)$. Then the convex sets $\P_1(G)$ and $\P_1(H)$ are affinely homeomorphic, where we use the topologies induced by the norm topologies of $\A(G)$ and $\A(H)$.
\end{prop}


\section{Convex structure of \texorpdfstring{$\P_1(G)$}{P1(G)} and structure of the group von Neumann algebra}
\label{sec-compact}

\paragraph{Topology on $\P_1(G)$.}
Assume that the group $G$ is compact. Then every function $\varphi \in \P_1(G)$ belongs to the Fourier algebra $\A(G)$ and satisfies $\norm{\varphi}_{\A(G)}=\varphi(e)=1$. Hence $\P_1(G)=\A(G)\cap \P_1(G)$ and by Proposition \ref{prop-first}, the convex set $\P_1(G)$ identifies affinely with the normal state space of the von Neumann algebra $\VN(G)$. We equip the convex set $\P_1(G)$ with the induced topology coming from the norm topology of $\A(G)$ (equivalently, the norm topology of the predual $\VN(G)_*$). When we speak of closed faces of $\P_1(G)$ in the sequel, we mean closed for this topology.

\paragraph{Facial structure of $\P_1(G)$ and von Neumann algebra $\VN(G)$}
Recall that if the group $G$ is compact, there exists a $*$-isomorphism
\begin{equation}
\label{VNG-as-sum-of-matricial-algebras}
\VN(G)
\cong \bigoplus_{\pi \in \Irr(G)} \M_{d_\pi},
\end{equation} 
where $\Irr(G)$ is the set of all equivalence classes of irreducible unitary representations.  See, e.g., \cite[Theorem 8.3.11 p.~53]{Rau17}. Here $d_\pi$ is the degree of $\pi \in \Irr(G)$ and $\bigoplus_{\pi \in \Irr(G)} \M_{d_\pi}$ is defined as the set $\{(x_{\pi})_{\pi \in \Irr(G)} : x_\pi \in \M_{d_\pi}, \sup_{\pi \in \Irr(G)} \norm{x_{\pi}} < \infty \}$ equipped with the canonical involution and product. We start with the following result.

\begin{lemma}
\label{lemma-maximal}
Let $G$ be a compact group. There exists a maximal family $(F_i)_{i \in I}$ of mutually disjoint minimal split nonempty closed faces of the convex set $\P_1(G)$.
\end{lemma}

\begin{proof}
We justify the existence of the family by Zorn's lemma, which states that if every totally ordered subset in a nonempty partially ordered set $P$ has an upper bound in $P$, then $P$ admits a maximal element. Let $\cal{S}$ be the set of all families $\cal{E}$ of mutually disjoint minimal split nonempty closed faces of $\P_1(G)$, partially ordered by inclusion.

Observe that the center of the matrix algebra $\M_{d_\pi}$ is $\mathbb{C}1$ by \cite[Example 5.1.7 p.~309]{KaR97a}. Using \cite[p.~21]{Dix81} in the second equality, we deduce that the center of $\VN(G)$ is  
\begin{equation}
\label{center}
\Z(\VN(G)) 
\ov{\eqref{VNG-as-sum-of-matricial-algebras}}{\cong} \Z \bigg(\bigoplus_{\pi \in \Irr(G)}\M_{d_\pi}\bigg)
=\bigoplus_{\pi \in \Irr(G)} \Z(\M_{d_\pi})
\cong \ell^\infty_{\Irr(G)}.
\end{equation}
Consequently, this center contains minimal projections, where minimality is understood relative to $\Z(\VN(G))$ rather than to $\VN(G)$. 

Fix a minimal projection $p \in \Z(\VN(G))$. By Proposition \ref{prop-Alfsen}, the set $\Face(p)$ is a split nonempty closed face of the convex set $\P_1(G)$, and it is minimal among split nonempty closed faces since $p$ is minimal among non-zero central projections, using the order preserving property \eqref{order} in Proposition \ref{prop-Alfsen}. Indeed, if $F$ is a split nonempty closed face with $F \subset \Face(p)$, then $F=\Face(q)$ for some central projection $q$ and $\Face(q)\subset \Face(p)$ implies $q \leq p$ by \eqref{order}. Since $p$ is minimal in $\Z(\VN(G))$, we have $q=0$ or $q=p$. As $F$ is nonempty, $q \neq 0$. Hence $q=p$ and $F=\Face(p)$. Therefore $\{\Face(p)\} \in \cal{S}$. So the set $\cal{S}$ is nonempty.

Now, let $(\cal{E}_j)_{j \in J}$ be a totally ordered subset of $\cal{S}$. Set
\[
\cal{E} \ov{\mathrm{def}}{=} \bigcup_{j \in J} \cal{E}_j.
\]
Then $\cal{E}$ is again a family of mutually disjoint minimal split nonempty closed faces (disjointness and minimality are preserved because the chain is totally ordered by inclusion). Hence $\cal{E}$ is an upper bound of $(\cal{E}_j)_{j \in J}$ in $\cal{S}$. By Zorn's lemma, $\cal{S}$ admits a maximal element, that is, a maximal family $(F_i)_{i \in I}$ of mutually disjoint minimal split nonempty closed faces of the set $\P_1(G)$.
\end{proof}

\begin{prop}
\label{prop:main-vNG}
Let $G$ be a compact group. Let $\P_1(G)$ be equipped with the topology induced by the norm
topology of $\A(G)$. The class of $\P_1(G)$ up to affine homeomorphisms determines the von Neumann algebra $\VN(G)$ up to $*$-isomorphism.
\end{prop}

\begin{proof}
By Lemma \ref{lemma-maximal}, there exists a maximal family $(F_i)_{i \in I}$ of mutually disjoint minimal split nonempty closed faces of the convex set $\P_1(G)$. 
By Proposition \ref{prop-Alfsen}, for each $i \in I$, there exists a unique central projection $p_i \in \VN(G)$, which is minimal in the center $\Z(\VN(G))$, such that
\begin{equation}
\label{inter-final}
F_i
=\Face(p_i)
\ov{\eqref{def-Face-p}}{=} \bigl\{\varphi\in\P_1(G): \Theta_G(\varphi)(p_i)=1 \bigr\}.
\end{equation}
Moreover, if $i \neq j$, the faces $F_i$ and $F_j$ are disjoint, i.e., $F_i \cap F_j =\emptyset$. Since $p_i$ and $p_j$ are central, they commute, and consequently $p_ip_j=p_i \wedge p_j$, according to \cite[Proposition 2.5.3 p.~111]{KaR97a}. Using Proposition \ref{prop-Alfsen}, we obtain
$$
\Face(p_i p_j)
=\Face(p_i \wedge p_j)
\ov{\eqref{disjoint}}{=} \Face(p_i) \cap \Face(p_j)
=\emptyset.
$$ 
Using again Proposition \ref{prop-Alfsen}, we conclude that $p_i p_j=0$. Thus $(p_i)_{i \in I}$ is a family of pairwise orthogonal minimal central projections.

Let $p \ov{\mathrm{def}}{=} \sum_{i \in I} p_i$, where the sum is taken in the strong operator topology. This sum exists because the projections $(p_i)_{i \in I}$ are pairwise orthogonal, see Proposition \ref{prop-sum-projections}. Since the center of a von Neumann algebra is strongly closed and each $p_i$ is central, the strong sum $p$ is a central projection. 

We claim that $p=1$. Assume by contradiction that $p \neq 1$ and set $r \ov{\mathrm{def}}{=} 1-p$. Then $r$ is a non-zero central projection. By the decomposition $
\VN(G)
\cong
\bigoplus_{\pi \in \Irr(G)} \M_{d_\pi}$ and Proposition \ref{prop-carac-atomic}, the von Neumann algebra $\VN(G)$ is atomic, and so is its center, again by Proposition \ref{prop-carac-atomic}. Hence $r$ dominates a minimal central projection $q \leq r$. By Proposition \ref{prop-Alfsen}, the face $\Face(q)$ is a nonempty split face, and it is minimal among nonempty split faces because $q$ is minimal and central. Furthermore, $q \leq r = 1-p$ implies $q p_i =0$ for all $i \in I$. Hence $\Face(q)$ is disjoint from each face $F_i$ by Proposition \ref{prop-Alfsen}. Therefore, the family $(F_i)_{i \in I} \cup \{\Face(q)\}$ is a family of mutually disjoint minimal split nonempty closed faces, contradicting the maximality of $(F_i)_{i \in I}$. Thus $r=0$, that is, $p=1$. In other words, we have
\begin{equation}
\label{sum=1}
\sum_{i \in I} p_i 
= 1,
\end{equation}
with strong convergence.

By Example \ref{example-minimal-2}, the minimal projections of the center $\Z(\VN(G))$, described in \eqref{center}, identify exactly the elements $e_\pi$ of the canonical basis of $\ell^\infty_{\Irr(G)}$ for $\pi \in \Irr(G)$. Since it is a sum of matrix algebras, the von Neumann algebra $\VN(G)$ is of type $\I$ by \cite[Proposition 6.7.2 p.~303]{Li92} or Proposition \ref{prop-carac-atomic}. We deduce by \cite[p.~422]{KaR97b} that each reduced von Neumann algebra $\VN(G)e_\pi$ is also of type $\I$. Furthermore, since $e_\pi$ is minimal in $\Z(\VN(G))$ the von Neumann algebra $\VN(G)e_\pi$ is a factor by \cite[p.~420]{KaR97b}.


Recall that by \cite[Theorem 6.6.1 p.~426]{KaR97b} a type $\I$ factor is $*$-isomorphic to $\B(\cal{H})$ for some Hilbert space $\cal{H}$ with dimension $\dim \cal{H}$ belonging to $\{1,2,\ldots\} \cup \{\text{infinite cardinals}\}$. So we have a $*$-isomorphism $\VN(G)e_\pi \cong \B(\cal{H}_\pi)$ with $\dim \cal{H}_\pi=d_\pi$. Since we have the (unique) decomposition \eqref{VNG-as-sum-of-matricial-algebras} into factors, every factor summand of $\VN(G)$ is finite-dimensional. 

Since $p_i$ is a non-zero central projection which is minimal in $\Z(\VN(G))$, there exists a unique $\pi_i \in \Irr(G)$ such that $p_i=e_{\pi_i}$. This implies that $\VN(G)p_i=\VN(G)e_{\pi_i} \cong \M_{d_{\pi_i}}$. We let $n_i \ov{\mathrm{def}}{=} d_{\pi_i}$. Fix, with \cite[Corollary 6.5.3 p.~424]{KaR97b}, a maximal family $(p_{i,1},\ldots,p_{i,n_i})$ of mutually orthogonal minimal projections in $\VN(G)p_i$ such that
\[
p_i
=\sum_{k=1}^{n_i} p_{i,k},
\]
where the sum is taken in the strong operator topology. For any integer $k \in \{1,\ldots, n_i\}$, set
\[
q_{i,k} 
\ov{\mathrm{def}}{=} p_{i,1}+\cdots+p_{i,k}.
\]
Then $(q_{i,k})_{k=1}^{n_i}$ is a strictly increasing family of projections in $\VN(G)p_i$, and by Proposition \ref{prop-Alfsen} we obtain a strictly increasing family of faces
\begin{equation}
\label{eq:facechain}
\Face(q_{i,1}) \subsetneq \Face(q_{i,2}) \subsetneq \cdots \subsetneq \Face(q_{i,n_i})
= \Face(p_i)
\ov{\eqref{inter-final}}{=} F_i.
\end{equation}
We claim that \eqref{eq:facechain} is maximal, and that $n_i$ is the maximal cardinality of a strictly increasing chain of nonempty norm-closed faces contained in $F_i$. Indeed, if
\[
E_1 \subsetneq \cdots \subsetneq E_m \subseteq F_i
\]
is a strictly increasing chain of nonempty norm-closed faces, then, by Proposition \ref{prop-Alfsen}, there exist non-zero projections
\[
0< f_1<\cdots<f_m\leq p_i
\]
such that $E_k=\Face(f_k)$ for any $1 \leq k \leq m$. Under a $*$-isomorphism $\VN(G)p_i \cong \M_{n_i}$, projections correspond to orthogonal projections onto subspaces of $\mathbb{C}^{n_i}$, hence each projection has a rank in $\{0,1,\ldots,n_i\}$. Then
\[
1 
\leq \rank(f_1) < \rank(f_2) < \cdots < \rank(f_m) \leq n_i,
\]
so necessarily $m \leq n_i$. Since the chain in \eqref{eq:facechain} has cardinality $n_i$, it is maximal. Consequently, the integer $n_i$ can be recovered from the facial structure of $F_i$ and is invariant under affine homeomorphism, by Lemma \ref{lem-affine-homeomorphism-preserves-face-chains}.




Let $\pi \in \Irr(G)$. Since $\sum_{i \in I} p_i \ov{\eqref{sum=1}}{=} 1$ strongly and since left multiplication by the fixed bounded operator $e_\pi$ is strongly continuous on bounded sets by \cite[Proposition 1.2.1 p.~9]{Li92}, we obtain
\begin{equation}
\label{fin-2345}
e_\pi
=\sum_{i \in I} e_\pi p_i
\end{equation}
in the strong operator topology. Now, each $e_\pi p_i$ is a projection by \cite[Proposition 2.5.3 p.~111]{KaR97a}, which is central and dominated by $e_\pi$. As $e_\pi$ is minimal in $\Z(\VN(G))$, for every $i$, we have $e_\pi p_i=0$ or $e_\pi p_i=e_\pi$. By \eqref{fin-2345}, since $e_\pi \neq 0$ and the projections $(p_i)_{i \in I}$ are pairwise orthogonal, there exists a unique $i \in I$ such that $e_\pi p_i=e_\pi$. Hence $e_\pi \leq p_i$, and since $p_i$ is also minimal central, we obtain $e_\pi=p_i$. Therefore the family $(p_i)_{i \in I}$ coincides with the family $(e_\pi)_{\pi \in \Irr(G)}$ up to reindexing. In particular, the map $i \mapsto \pi_i$ is a bijection between $I$ and $\Irr(G)$.
Thus, knowing the facial structure of $\P_1(G)$ allows us to recover the decomposition \eqref{VNG-as-sum-of-matricial-algebras} up to $*$-isomorphism.
\end{proof}

\begin{cor}
\label{cor-abelian-gp}
Let $G$ and $H$ be compact groups. Suppose that the convex sets $\P_1(G)$ and $\P_1(H)$ are affinely homeomorphic, where we use the topologies induced by the norm topologies of $\A(G)$ and $\A(H)$. Then the von Neumann algebras $\VN(G)$ and $\VN(H)$ are $*$-isomorphic.
\end{cor}

\begin{proof}
It suffices to use Proposition \ref{prop:main-vNG}.
\end{proof}

A combination of Corollary \ref{cor-abelian-gp} and Proposition \ref{prop:compact-converse} yields the following characterization.

\begin{thm}
\label{main}
Let $G$ and $H$ be compact groups. The von Neumann algebras $\VN(G)$ and $\VN(H)$ are $*$-isomorphic if and only if the convex sets $\P_1(G)$ and $\P_1(H)$ are affinely homeomorphic, where we use the topologies induced by the norm topologies of $\A(G)$ and $\A(H)$.
\end{thm}

\begin{remark} \normalfont
Let $G$ be a locally compact group. Recall that by \cite[Theorem C.5.2 p.~358]{BHV08}, a normalized positive definite function $\varphi$ is an extreme point of the convex set \texorpdfstring{$\P_1(G)$}{P1(G)} if and only if the associated cyclic GNS unitary representation $\pi_\varphi$ is irreducible. 
\end{remark}

\section{Affine homeomorphisms of the set $\P_1(G)$ of normalized positive definite functions}
\label{sec-affine-homeo}

Here, we describe the group of affine homeomorphisms of the set $\P_1(G)$ of normalized positive definite functions when the group $G$ is compact. 

\paragraph{Jordan $*$-homomorphisms} A Jordan $*$-homomorphism \cite[p.~8]{Sto13} between von Neumann algebras $\cal{M}$ and $\cal{N}$ is a linear map $J \co \cal{M} \to \cal{N}$ such that
$$
J(x^*) =J(x)^*
\quad \text{and} \quad 
J(xy+yx)
=J(x)J(y)+J(y)J(x), \quad x,y \in \cal{M}.
$$
This means in particular that $J$ preserves the Jordan product $x \circ y \ov{\mathrm{def}}{=} \frac{1}{2}(xy+yx)$. It is worth noting that $*$-homomorphisms and $*$-anti-homomorphisms are Jordan $*$-homomorphisms. Furthermore, according to \cite[Result 1 p.~156]{Emch09} a Jordan $*$-automorphism $J \co \cal{M} \to \cal{M}$ is necessarily normal. Moreover, by \cite[Theorem 3.2.4 p.~32]{Sto13} (see also \cite[Exercise 10.5.27 p.~777]{KaR97b} for a more general result), a Jordan $*$-automorphism $J \co \B(H)\to \B(H)$ of $\B(H)$, where $H$ is a complex Hilbert space, is either a $*$-automorphism or a $*$-anti-automorphism of $\B(H)$. 
More explicitly, by \cite[Theorem 3.2.4 p.~32]{Sto13} (see also \cite[p.~145]{Lan17} for a different presentation), there exists a unitary $u \in \B(H)$ such that, in the first case,
\begin{equation}
\label{unit-1}
J(x)
=uxu^*,
\quad x \in \B(H),
\end{equation}
whereas, in the second case,
\begin{equation}
\label{unit-2}
J(x)
=u x^\top u^*,
\quad x \in \B(H),
\end{equation}
where $\top$ denotes the transpose operation with respect to a fixed orthonormal basis of $H$.


\paragraph{The group of affine homeomorphisms of $\P_1(G)$}

If $G$ is a locally compact group, recall that we have an isometric isomorphism $\Theta_G \co \A(G) \to \VN(G)_*$ of Banach spaces.

\begin{prop}
\label{prop:affhomeo-P1-compact}
Let $G$ be a compact group. Equip the convex set $\P_1(G)$ with the topology induced by the norm topology of $\A(G)$. Let $T \co \P_1(G) \to \P_1(G)$ be an affine homeomorphism. Then there exists a unique Jordan $*$-automorphism $\theta \co \VN(G) \to \VN(G)$ such that
\begin{equation}
\label{eq:T-by-theta-lambda}
(T\varphi)(s)
=
\Theta_G(\varphi)\big(\theta(\lambda_s)\big), \quad \varphi \in \P_1(G),s \in G.
\end{equation}
Moreover, if we write the decomposition $\VN(G)\cong \bigoplus_{\pi \in \Irr(G)} \M_{d_\pi}$ and denote by $p_\pi$ the minimal central projection supporting the summand $\M_{d_\pi}$, then there exists a permutation $\sigma$ of $\Irr(G)$ such that $d_{\sigma(\pi)}=d_\pi$ for all $\pi \in \Irr(G)$, and for each $\pi$ there exists a unitary $u_\pi \in \M_{d_\pi}$ such that the restriction of $\theta$ to $\VN(G)p_\pi \cong \M_{d_\pi}$ is either
\[
x \mapsto u_\pi x u_\pi^*,
\quad \text{or} \quad
x \mapsto u_\pi x^{\top} u_\pi^*,
\quad x \in \M_{d_\pi},
\]
after using the identification $\VN(G)p_{\sigma(\pi)} \cong \M_{d_\pi}$. Conversely, any Jordan $*$-automorphism $\theta \co \VN(G) \to \VN(G)$ gives rise, through the construction \eqref{eq:T-by-theta-lambda}, to an affine homeomorphism of $\P_1(G)$.
\end{prop}

\begin{proof}
From \eqref{compact-case} and Proposition \ref{prop-first}, the convex set $\P_1(G)$ identifies affinely with the normal state space
\[
\cal{K}
\ov{\mathrm{def}}{=}
\{\omega \in \VN(G)_*^+ : \omega(1)=1\}
\]
via the map $\Theta_G$. In particular, the given topology on $\P_1(G)$ agrees with the norm topology on the subspace $\cal{K}$ of $\VN(G)_*$. Transporting $T \co \P_1(G) \to \P_1(G)$ through $\Theta_G \co \P_1(G) \to \cal{K}$, we obtain an affine homeomorphism
\begin{equation}
\label{final-66}
\widetilde{T}
\ov{\mathrm{def}}{=}
\Theta_G \circ T \circ \Theta_G^{-1}
\co
\cal{K} \to \cal{K}.
\end{equation}
By \cite[Proposition 5.16 p.~147]{AlS03}, there exists a unique Jordan $*$-automorphism $\theta \co \VN(G) \to \VN(G)$  such that $\widetilde{T}=\theta_*|_{\cal{K}}$. This means that
\[
\widetilde{T}(\omega)
=
\omega \circ \theta,
\quad
\omega \in \cal{K}.
\]
Using \eqref{final-66}, we deduce that $(\Theta_G \circ T \circ \Theta_G^{-1})(\omega)=\omega \circ \theta$.  Transporting this identity back through the affine homeomorphism $\Theta_G \co \P_1(G) \to \{\omega \in \VN(G)_*^+ : \omega(1)=1\}$ gives 
\begin{equation}
\label{eq:T-by-theta}
(\Theta_G \circ T)(\varphi)
=
\Theta_G(\varphi)\circ \theta,
\quad
\varphi \in \P_1(G).
\end{equation}
For any $s \in G$, observe that we have $((\Theta_G \circ T)(\varphi))(\lambda_s) \ov{\eqref{Eymard-predual}}{=}(T\varphi)(s)$. So the equality \eqref{eq:T-by-theta} is equivalent to the equality \eqref{eq:T-by-theta-lambda}, i.e.,
\begin{equation*}
(T\varphi)(s)
=
\Theta_G(\varphi)\big(\theta(\lambda_s)\big), \quad \varphi \in \P_1(G),s \in G.
\end{equation*}
By \cite[Exercise 10.5.22 (viii) p.~774]{KaR97b}, it permutes the minimal central projections. Thus there exists a permutation $\sigma$ of $\Irr(G)$ such that $\theta(p_\pi)=p_{\sigma(\pi)}$. Since $p_\pi$ is central, \cite[Exercise 10.5.22 (viii), p.~774]{KaR97b} yields $\theta(xp_\pi) = \theta(x)\theta(p_\pi)$ for any $x \in \VN(G)$.
Hence, we obtain
$$
\theta(\VN(G)p_\pi)
=\theta(\VN(G))\theta(p_\pi)
=\VN(G)p_{\sigma(\pi)}.
$$
Consequently, $\theta$ restricts to a Jordan $*$-isomorphism $
\theta_\pi
\co
\VN(G)p_\pi
\to
\VN(G)p_{\sigma(\pi)}$. Since these algebras are isomorphic to $\M_{d_\pi}$ and $\M_{d_{\sigma(\pi)}}$, respectively, their dimensions coincide, so $d_{\sigma(\pi)}^2 = d_\pi^2$ and hence $d_{\sigma(\pi)}=d_\pi$. Finally, by \cite[Exercise 10.5.27 p.~777]{KaR97b}, a Jordan $*$-isomorphism between full matrix algebras is either a $*$-isomorphism or a $*$-anti-isomorphism. Therefore, after fixing identifications of both blocks with $\M_{d_\pi}$, there exists a unitary $u_\pi$ such that
\[
\theta_\pi(x)=u_\pi xu_\pi^*
\quad
\text{or}
\quad
\theta_\pi(x)
=u_\pi x^\top u_\pi^*.
\]

Conversely, let $\theta \co \VN(G) \to \VN(G)$ be a normal Jordan $*$-automorphism of $\VN(G)$. Since $\theta$ and $\theta^{-1}$ are unital by \cite[Exercise 10.5.22 p.~774]{KaR97b}, positive by \cite[Exercise 10.5.31 p.~778]{KaR97b} and normal by \cite[Result 1 p.~156]{Emch09}, the map $S \co \cal{K} \to \cal{K}$ defined by
$$
S(\omega)
=\omega \circ \theta, \quad \omega \in \cal{K},
$$
is an affine norm homeomorphism, whose inverse is $\cal{K} \to \cal{K}$, $\omega \mapsto \omega \circ \theta^{-1}$. Therefore $
T
\ov{\mathrm{def}}{=}
\Theta_G^{-1} \circ S \circ \Theta_G$ is an affine homeomorphism of $\P_1(G)$ and satisfies \eqref{eq:T-by-theta-lambda}.
\end{proof}


\begin{remark}\normalfont
The group $\mathrm{JAut}(\VN(G))$ of Jordan $*$-automorphisms of the group von Neumann algebra $\VN(G)$ is isomorphic to the group $\mathrm{AffHomeo}(\P_1(G))$ of affine homeomorphisms of $\P_1(G)$. More precisely, the map
$\mathrm{JAut}(\VN(G)) \to \mathrm{AffHomeo}(\P_1(G))$, $\theta \mapsto T_\theta$, where
\[
(T_\theta\varphi)(s)
\ov{\mathrm{def}}{=}
\Theta_G(\varphi)\bigl(\theta^{-1}(\lambda_s)\bigr),
\qquad
\varphi\in\P_1(G),\ s\in G,
\]
is a group isomorphism. The occurrence of $\theta^{-1}$ in this definition is essential for obtaining a group isomorphism rather than an anti-isomorphism. Indeed, precomposition reverses the order of composition, and therefore
\[
T_{\theta_1}\circ T_{\theta_2}=T_{\theta_1\circ\theta_2}, \quad \theta_1,\theta_2\in\mathrm{JAut}(\VN(G)).
\]
Without the inverse, one would instead obtain $T_{\theta_1}\circ T_{\theta_2}=T_{\theta_2\circ\theta_1}$.
\end{remark}

\paragraph{Declaration of interest} None.

\paragraph{Competing interests} The authors declare that they have no competing interests.

\paragraph{Data availability} No datasets were generated during this study.

\paragraph{Acknowledgment} 
Part of this work was carried out during a visit of the first author to Nankai University (China) in 2025, supported by the National Natural Science Foundation of China (Grant No.~12571143). The first author gratefully acknowledges the hospitality of Nankai University.




{\footnotesize

\vspace{0.2cm}

\noindent C\'edric Arhancet\\ 
\noindent 6 rue Didier Daurat, 81000 Albi, France\\
URL: \href{http://sites.google.com/site/cedricarhancet}{https://sites.google.com/site/cedricarhancet}\\
cedric.arhancet@protonmail.com\\
ORCID: 0000-0002-5179-6972 

\vspace{0.2cm}

\noindent Lei Li \\
\noindent School of Mathematical Sciences and LPMC \\
Nankai University, 300071 Tianjin, China \\
leilee@nankai.edu.cn
}

\normalsize


\small
\begin{thebibliography}{79}


\bibitem[AlS01]{AlS01}
E. M. Alfsen and F. W. Shultz.
\newblock State spaces of operator algebras. Basic theory, orientations, and $C^*$-products.
\newblock Math. Theory Appl. Birkh\"auser Boston, Inc., Boston, MA, 2001.

\bibitem[AlS03]{AlS03}
E. M. Alfsen and F. W. Shultz.
\newblock Geometry of state spaces of operator algebras. 
\newblock Mathematics: Theory \& Applications. Birkh\"auser Boston, Inc., Boston, MA, 2003. 

\bibitem[AlS76]{AlS76}
E. Alfsen and F. W. Shultz.
\newblock Non-commutative spectral theory for affine function spaces on convex sets.
\newblock Mem. Amer. Math. Soc. 6 (1976), no. 172.

\bibitem[Arh25]{Arh25}
C. Arhancet.
\newblock Quantum information theory via Fourier multipliers on quantum groups.
\newblock Preprint, arXiv:2008.12019.

\bibitem[BHV08]{BHV08}
B. Bekka, P. de la Harpe, and A. Valette.
\newblock Kazhdan's property (T).
\newblock New Mathematical Monographs, 11. Cambridge University Press, Cambridge, 2008.

\bibitem[Bla06]{Bla06}
B. Blackadar.
\newblock Operator algebras. Theory of $C^*$-algebras and von Neumann algebras.
\newblock Encyclopaedia of Mathematical Sciences, 122. Operator Algebras and Non-commutative Geometry, III. Springer-Verlag, Berlin, 2006.

\bibitem[BrO08]{BrO08}
N. P. Brown and N. Ozawa.
\newblock $C^*$-algebras and finite-dimensional approximations.
\newblock Graduate Studies in Mathematics, 88. American Mathematical Society, Providence, RI, 2008.


\bibitem[Dix77]{Dix77}
J. Dixmier.
\newblock $C^*$-algebras. Translated from the French by Francis Jellett.
\newblock North-Holland Mathematical Library, Vol. 15. North-Holland Publishing Co., Amsterdam-New York-Oxford, 1977. 

\bibitem[Dix81]{Dix81}
J. Dixmier.
\newblock Von Neumann algebras. With a preface by E. C. Lance. Translated from the second French edition by F. Jellett.
\newblock North-Holland Mathematical Library, 27. North-Holland Publishing Co., Amsterdam-New York, 1981.
 
\bibitem[Emch09]{Emch09}
G. G. Emch.
\newblock Algebraic Methods in Statistical Mechanics and Quantum Field Theory.
\newblock Dover Publications, 2009.

\bibitem[Eym64]{Eym64}
P. Eymard.
\newblock L'alg\`ebre de Fourier d'un groupe localement compact (French).
\newblock  Bull. Soc. Math. France 92 (1964), 181--236.

\bibitem[Hel09]{Hel09}
M. Hellmich.
\newblock Decoherence in Infinite Quantum Systems.
\newblock PhD Thesis, Universit\"at Bielefeld, 2009.

\bibitem[KaR97a]{KaR97a}
R. V. Kadison and J. R. Ringrose.
\newblock Fundamentals of the theory of operator algebras. Vol. I. Elementary theory. Reprint of the 1983 original. 
\newblock Graduate Studies in Mathematics, 15. American Mathematical Society, Providence, RI, 1997. 

\bibitem[KaR97b]{KaR97b}
R. V. Kadison and J. R. Ringrose.
\newblock Fundamentals of the theory of operator algebras. Vol. II. Advanced theory. Corrected reprint of the 1986 original.
\newblock Graduate Studies in Mathematics, 16. American Mathematical Society, Providence, RI, 1997.

\bibitem[KaL18]{KaL18}
E. Kaniuth and A. T.-M. Lau.
\newblock Fourier and Fourier-Stieltjes algebras on locally compact groups. 
\newblock Mathematical Surveys and Monographs, 231. American Mathematical Society, Providence, RI, 2018. 

\bibitem[Lan17]{Lan17}%
K. Landsman.
\newblock Foundations of Quantum Theory: From Classical Concepts to Operator Algebras.
\newblock Fundamental Theories of Physics, vol. 188. Springer, Cham, 2017.

\bibitem[LWW22]{LWW22}
L. Li, Y.-S. Wang and Y.-L. Wong.
\newblock How the norm one positive definite functions determine a finite group.
\newblock Linear Algebra Appl. 633 (2022), 71--103.

\bibitem[Li92]{Li92}
B. R. Li.
\newblock Introduction to operator algebras.
\newblock World Scientific Publishing Co., Inc., River Edge, NJ, 1992.


\bibitem[Rau17]{Rau17}
S. Raum.
\newblock Lecture notes on abstract harmonic analysis.
\newblock Available online, 2017. \href{https://raum-brothers.eu/sven/data/teaching/2016-17/AHA/aha-lecture-notes-2017-05-18.pdf}{\nolinkurl{https://raum-brothers.eu/sven/data/teaching/2016-17/AHA/aha-lecture-notes-2017-05-18.pdf}}

\bibitem[Sto13]{Sto13}
E. St\o rmer.
\newblock Positive linear maps of operator algebras.
\newblock Springer Monographs in Mathematics. Springer, Heidelberg, 2013.

\bibitem[Sun97]{Sun97}
V. S. Sunder.
\newblock Functional analysis. Spectral theory.
\newblock Birkh\"auser Adv. Texts Basler Lehrb\"ucher. Birkh\"auser Verlag, Basel, 1997.

\bibitem[Tak02]{Tak02}
M. Takesaki.
\newblock Theory of operator algebras. I. Reprint of the first (1979) edition.
\newblock Encyclopaedia of Mathematical Sciences, 124. Operator Algebras and Non-commutative Geometry, 5. Springer-Verlag, Berlin, 2002.

\end{thebibliography}
\end{document}